\documentclass[11pt]{amsart}
\usepackage{amssymb}
\usepackage{enumerate}
\usepackage[colorlinks=true,linkcolor=blue,citecolor=blue,urlcolor=blue]{hyperref}

\newtheorem{mainthm}{Theorem}

\newtheorem{theorem}{Theorem}[section]
\newtheorem{lemma}[theorem]{Lemma}
\newtheorem{proposition}[theorem]{Proposition}
\newtheorem{corollary}[theorem]{Corollary}
\theoremstyle{definition}
\newtheorem{definition}[theorem]{Definition}
\newtheorem{example}[theorem]{Example}
\newtheorem{remark}[theorem]{Remark}
\newtheorem{question}[theorem]{Question}

\newcommand{\F}{\mathbb{F}}
\newcommand{\K}{\mathbb{K}}

\newcommand{\GL}{\mathrm{GL}}
\newcommand{\GaL}{\Gamma\mathrm{L}}
\newcommand{\SD}{\mathrm{SD}}
\newcommand{\Sym}{\mathrm{Sym}}
\newcommand{\Stab}{\mathrm{Stab}}
\newcommand{\End}{\mathrm{End}}
\newcommand{\Ind}{\mathrm{Ind}}
\newcommand{\sym}{\operatorname{sym}}
\newcommand{\cO}{\mathcal{O}}
\newcommand{\cA}{\mathcal{A}}
\newcommand{\xor}{\veebar}
\newcommand{\diag}{\operatorname{diag}}

\begin{document}

\title[$2$-groups with exactly one regular orbit]{Linear $2$-group actions with exactly one regular orbit}

\author{Thomas Michael Keller}
\address{Department of Mathematics, Texas State University, 601 University Drive, San Marcos, TX 78666, USA}
\email{keller@txstate.edu}

\author{Jeffrey M. Riedl}
\address{Department of Mathematics, University of Akron, Akron, OH, USA}
\email{riedl@uakron.edu}

\subjclass[2020]{Primary 20D15; Secondary 20C20, 20B35}
\keywords{Linear group action, regular orbit, finite $2$-group, wreath product, Fermat prime}

\begin{abstract}
Let $G$ be a finite $2$-group and $V$ a finite faithful irreducible $G$-module. We show that if $G$ has exactly one regular orbit on $V$, then the characteristic $q$ of $V$ is a Fermat prime, $V$ is irreducible over $\F_q$, $\dim_{\F_q}V$ is a power of $2$, and the centraliser of $G$ in $\GL(V)$ embeds into $G$. For every Fermat prime $q$ and every $m\ge1$ we determine the largest orbit size $2^{x_m}$ of a Sylow $2$-subgroup $P_m$ of $\GL_{2^{m-1}}(q)$ on its natural module $V_m$, and the number of orbits of this size, which is $1$, $2$ or $3$. Our main result is that for every Fermat prime $q$ and every $m\ge1$ there is an irreducible subgroup of $P_m$ of order $2^{x_m}$ that has exactly one regular orbit on $V_m$. In this respect $p=2$ behaves like $p=3$, and unlike the Mersenne primes $p\ge7$ studied in an earlier paper. The proof places a group related to a Singer cycle into every block of $q-1$ coordinates. This reduces the problem, for all Fermat primes at once, to a question about subgroups of a Sylow $2$-subgroup of a symmetric group acting on subsets, which we answer by means of fibre products. We also determine which largest orbits occur in two small levels, and we classify the examples in some small cases with the help of a computer.
\end{abstract}

\maketitle

\section{Introduction}

Orbit structures of finite linear groups, and in particular the existence of regular orbits, play an important role in the representation theory of finite groups; see, for instance, \cite{CY,MW}. In \cite{KR} we initiated the study of the extreme situation of exactly one regular orbit. There we considered pairs $(G,V)$, where $G$ is a finite $p$-group for an odd prime $p$ and $V$ is a finite faithful irreducible $G$-module such that $G$ has exactly one regular orbit on $V$. We showed that then $p$ is a Mersenne prime and that $(G,V)$ embeds into $(P_m,V_m)$, where $P_m$ is a Sylow $p$-subgroup of $\GL_1(p+1)\wr S_{p^{m-1}}$ and $V_m$ is its natural module, regarded over $\F_2$. We determined the largest orbit size $p^{x_m}$ of $P_m$ on $V_m$, and we showed that subgroups of $P_m$ of order $p^{x_m}$ with exactly one regular orbit exist for every $m$ if $p=3$, but only for $m\le3$ if $p\ge7$. In the present paper we treat the remaining prime $p=2$.

For $p=2$ the Mersenne primes are replaced by the Fermat primes, which now occur as the characteristic of $V$. To state our results, let $q$ be a Fermat prime. If $q\ge5$, we write $q=2^n+1$; for $m\ge1$ we let $P_m$ be a Sylow $2$-subgroup of $\GL_1(q)\wr S_{2^{m-1}}\le\GL_{2^{m-1}}(q)$, let $V_m=\F_q^{2^{m-1}}$ be its natural module, and put $k=m-1$. If $q=3$, the relevant field turns out to be $\F_9$, whose order is $2^3+1$. Accordingly, for $m\ge2$ we let $P_m$ be a Sylow $2$-subgroup of $\GaL_1(9)\wr S_{2^{m-2}}$ acting naturally on $V_m=\F_9^{2^{m-2}}$, which we regard as a vector space over $\F_3$, and we put $n=3$ and $k=m-2$; finally, $P_1=\GL_1(3)$ and $V_1=\F_3$. In all cases $\dim_{\F_q}V_m=2^{m-1}$, and $P_m$ is a Sylow $2$-subgroup of $\GL(V_m)$ (Lemma~\ref{lem:sylow}). We call $k$ the \emph{level}; it is the height of the binary tree underlying the permutation part of $P_m$.

\begin{mainthm}\label{thmA}
Let $G\ne1$ be a $2$-group, let $F$ be a finite field, and let $V$ be a faithful irreducible $FG$-module such that $G$ has exactly one regular orbit on $V$. Let $q$ be the characteristic of $F$, and let $\GL(V)$ denote the group of all $\F_q$-linear automorphisms of $V$. Then:
\begin{enumerate}[(a)]
\item $C_{\GL(V)}(G)$ is isomorphic to a subgroup of $G$;
\item $q$ is a Fermat prime, $V$ is irreducible as an $\F_qG$-module, and $\End_{\F_qG}(V)$ is $\F_q$ or, if $q=3$, possibly $\F_9$;
\item $\dim_{\F_q}V=2^{m-1}$ for some $m\ge1$, and $(G,V)$ is isomorphic to $(T,V_m)$ for some $T\le P_m$;
\item if $q\ge5$, then $V$ is induced from a one-dimensional module of a subgroup of $G$; if $q=3$, then $V\cong\Ind_H^G(V_1)$ for a subgroup $H$ and a module $V_1$ on which $H/C_H(V_1)$ acts primitively and is isomorphic to $C_2$ (with $\dim V_1=1$) or to $C_8$, $Q_8$ or $\SD_{16}$ (with $\dim V_1=2$).
\end{enumerate}
\end{mainthm}

In contrast to the case of odd $p$ \cite[proof of Theorem~4.2]{KR}, the primitive constituent of $V$ need not have exactly one regular orbit (Example~\ref{ex:wreath}). The reason is that, unlike permutation groups of odd order \cite{Gl}, permutation $2$-groups need not have regular orbits on the power set of the set they act on; the Sylow $2$-subgroup of $S_4$ is an example.

\begin{mainthm}\label{thmB}
Let $q$ be a Fermat prime and $m\ge1$, where $m\ge2$ if $q=3$, and let $n$ and $k$ be as above. Then the largest orbit size of $P_m$ on $V_m$ is $2^{x_m}$, where
\[
x_m=\begin{cases} n\,2^k, & k\le n,\\ n\,2^{n+1}+1, & k=n+1,\\ n\,2^k+3\cdot2^{k-n-2}-1, & k\ge n+2,\end{cases}
\]
and the number of orbits of this size is $1$ if $k<n$, $2$ if $k=n$, $1$ if $k=n+1$, and $3$ if $k\ge n+2$. In particular $x_{m+1}=2x_m+1$ whenever $k\ge n+2$.
\end{mainthm}

The largest orbits are described explicitly in Theorem~\ref{thm:largest}. With respect to the monomial basis of $V_m$ (over $\F_9$ if $q=3$), the unique largest orbit consists of the vectors without zero coordinates if $k<n$, and of the vectors with exactly one zero coordinate if $k=n+1$; for $k=n$ both of these orbits are largest orbits. For $k\ge n+2$ the three largest orbits $\cO^-$, $\cO^0$ and $\cO^+$ consist of vectors with exactly $2^{k-n-1}-1$, $2^{k-n-1}$ and $2^{k-n-1}+1$ zero coordinates, respectively, arranged in the patterns of Definition~\ref{def:Y}.

\begin{mainthm}\label{thmC}
Let $q$ be a Fermat prime and $m\ge1$. Then $P_m$ has an irreducible subgroup $T$ of order $2^{x_m}$ that has exactly one regular orbit on $V_m$. If $k\ge n+3$, then $T$ can be chosen such that this orbit is $\cO^0$.
\end{mainthm}

Together with \cite{KR}, Theorem~\ref{thmC} shows that subgroups of maximal order with exactly one regular orbit exist in every dimension for the primes $2$ and $3$, while for the Mersenne primes $p\ge7$ they exist only for $m\le3$ \cite[Theorem~E]{KR}.

\begin{mainthm}\label{thmD}
\begin{enumerate}[(a)]
\item Let $q\ge5$. The subgroups of $P_2=\GL_1(q)\wr C_2$ of order $(q-1)^2$ with exactly one regular orbit on $V_2$ are the base group and the group $M_1$ of Proposition~\ref{prop:m2}; only $M_1$ is irreducible. For $q=3$ these subgroups are $C_8$ and $Q_8$, and both are irreducible.
\item For $(q,m)=(3,2)$, $(5,2)$, $(17,2)$ and $(3,3)$ there are exactly $3$, $3$, $5$ and $34$ conjugacy classes of $2$-subgroups of $\GL(V_m)$ with exactly one regular orbit on $V_m$, of which $2$, $2$, $4$ and $10$ are irreducible (Table~\ref{tab:small}). For $(q,m)=(5,3)$ there are exactly $30$ conjugacy classes of subgroups of $\GL(V_3)$ of order $2^{x_3}=2^8$ with exactly one regular orbit, of which $16$ are irreducible; the corresponding $36$ conjugacy classes of subgroups of $P_3$ split evenly between the two largest orbits.
\end{enumerate}
\end{mainthm}

Part~(b) was obtained with the help of the computer algebra system GAP \cite{GAP}; see Section~\ref{sec:small}. As for $p=3$ \cite[Theorem~D]{KR}, a subgroup of order $2^{x_m}$ with exactly one regular orbit is in general not determined by this orbit (Remark~\ref{rem:nonunique}), and the even split in~(b) parallels the even splits found in \cite[Theorems~4.19 and~4.21]{KR}.

\begin{mainthm}\label{thmE}
Let $q$ be a Fermat prime and let $m$ be such that $k\in\{n+2,n+3\}$.
\begin{enumerate}[(a)]
\item If $k=n+2$, then each of $\cO^-$, $\cO^0$ and $\cO^+$ is the unique regular orbit of a subgroup of $P_m$ of order $2^{x_m}$; for $\cO^-$ and $\cO^+$ this subgroup can be chosen to be irreducible.
\item If $k=n+3$, then each of $\cO^-$ and $\cO^0$ is the unique regular orbit of an irreducible subgroup of $P_m$ of order $2^{x_m}$.
\end{enumerate}
\end{mainthm}

We do not know whether $\cO^-$ or $\cO^+$ is the unique regular orbit of a subgroup of order $2^{x_m}$ when $k\ge n+4$, nor whether $\cO^0$ is the unique regular orbit of an irreducible such subgroup when $k=n+2$; see Section~\ref{sec:open}.

\medskip
\noindent\textit{Methods.} The tools of \cite{KR} do not carry over directly to $p=2$. Besides the failure of Gluck's theorem mentioned above, the diagonal subgroup of $P_m$ is not elementary abelian if $q\ge5$, and the centralisers of vectors in largest orbits are spread over many layers of the iterated wreath product, not only over the bottom two. Our proofs are based instead on explicit constructions. The key idea for Theorem~\ref{thmC} is to place a group $\Sigma$, generated by the diagonal matrices of determinant $1$ and a monomial matrix of determinant $\zeta$ whose underlying permutation is a $2^n$-cycle, into every block of $2^n$ coordinates; here $\zeta$ generates the multiplicative group of the field. Since $\Sigma$ is regular on both largest orbits at level $n$, this reduces Theorem~\ref{thmC}, for all Fermat primes at once, to a combinatorial question about subgroups of the Sylow $2$-subgroup $W_j$ of $\Sym(2^j)$ acting on subsets of a set with $2^j$ elements (Proposition~\ref{prop:reduction}). We answer this question by means of fibre products (Proposition~\ref{prop:fibre}). All proofs are elementary; the computer was used only for Theorem~\ref{thmD}(b) and for some checks.

\medskip
\noindent\textit{Organisation.} Section~\ref{sec:prelim} collects notation and basic lemmas. Theorems~\ref{thmA}, \ref{thmB}, \ref{thmC}, \ref{thmE} and \ref{thmD} are proved in Sections~\ref{sec:A}, \ref{sec:B}, \ref{sec:C}, \ref{sec:E} and \ref{sec:small}, respectively. Section~\ref{sec:below} contains examples below the maximal order, and Section~\ref{sec:open} lists open problems.

\section{Notation and preliminaries}\label{sec:prelim}

All groups act from the right. For a group $G$ acting on a set $\Omega$ and $\Gamma\subseteq\Omega$ we write $\Stab_G(\Gamma)$ for the set stabiliser of $\Gamma$; for a vector $v$ of a $G$-module, $C_G(v)$ is its stabiliser. An orbit is \emph{regular} if its point stabilisers are trivial.

\subsection{Binary trees}
For $k\ge0$ let $\Omega_k=\{0,1,\dots,2^k-1\}$, and write $x=\sum_{i=0}^{k-1}x_i2^i$ with digits $x_i\in\{0,1\}$. For $x,c\in\Omega_k$ let $x\xor c$ be the element whose digits are $x_i+c_i\bmod2$. For $0\le i\le k$ the residue classes modulo $2^i$ partition $\Omega_k$ into $2^i$ blocks of size $2^{k-i}$, and every block of this partition is the union of two blocks of the next one. Let $W_k$ be the group of all permutations of $\Omega_k$ that map every residue class modulo $2^i$ onto a residue class modulo $2^i$, for all $0\le i\le k$.

Let $k\ge1$. The two residue classes modulo $2$ (the even and the odd numbers) are the \emph{halves} of $\Omega_k$. For $u\in W_{k-1}$ we define $u^L,u^R\in\Sym(\Omega_k)$ by
\[(2y)^{u^L}=2y^u,\quad (2y+1)^{u^L}=2y+1,\qquad (2y)^{u^R}=2y,\quad (2y+1)^{u^R}=2y^u+1\]
for $y\in\Omega_{k-1}$,
and we let $\tau=\tau_k$ be the involution $x\mapsto x\xor1$, which interchanges $2y$ and $2y+1$. For $Y,Y'\subseteq\Omega_{k-1}$ we put
\[(Y,Y')=\{2y:y\in Y\}\cup\{2y'+1:y'\in Y'\}\subseteq\Omega_k,\]
and for $X\subseteq\Omega_k$ we define $X_L,X_R\subseteq\Omega_{k-1}$ by $X=(X_L,X_R)$.

\begin{lemma}\label{lem:tree}
\begin{enumerate}[(a)]
\item Let $k\ge1$. Every element of $W_k$ can be written uniquely as $u_1^Lu_2^R\tau^\varepsilon$ with $u_1,u_2\in W_{k-1}$ and $\varepsilon\in\{0,1\}$, and all such products lie in $W_k$. Consequently $|W_k|=2^{2^k-1}$, and $W_k$ is a Sylow $2$-subgroup of $\Sym(\Omega_k)$ (cf.\ \cite{Ka}).
\item The $2^k$-cycle $g_k\colon x\mapsto x+1\bmod2^k$ and the elementary abelian group $E_k=\{x\mapsto x\xor c: c\in\Omega_k\}$ lie in $W_k$, and both are transitive on $\Omega_k$.
\item Let $k=a+j$ with $a,j\ge0$, and for $\beta\in\Omega_j$ let $B_\beta=\{\beta+2^j\ell:\ell\in\Omega_a\}$ be the residue class of $\beta$ modulo $2^j$. For $u\in W_a$ let $u_{(\beta)}$ act on $B_\beta$ by $\beta+2^j\ell\mapsto\beta+2^j\ell^u$ and trivially on $\Omega_k\setminus B_\beta$, and for $r\in W_j$ let $\hat r$ be the permutation $\beta+2^j\ell\mapsto\beta^r+2^j\ell$. Then all $u_{(\beta)}$ and $\hat r$ lie in $W_k$, the elements $u_{(\beta)}$ generate the stabiliser $B$ in $W_k$ of all the blocks $B_\beta$, $\hat W_j=\{\hat r:r\in W_j\}$ is a complement to $B$ in $W_k$, and $\hat r^{-1}u_{(\beta)}\hat r=u_{(\beta^r)}$.
\end{enumerate}
\end{lemma}

\begin{proof}
(a) An element $w\in W_k$ maps the halves to halves; replacing $w$ by $w\tau$ if necessary, we may assume that it fixes both halves. Then $w$ induces permutations $u_1\colon y\mapsto (2y)^w/2$ and $u_2\colon y\mapsto((2y+1)^w-1)/2$ of $\Omega_{k-1}$. As the residue class of $2y$ (or $2y+1$) modulo $2^i$ corresponds to the residue class of $y$ modulo $2^{i-1}$, we have $u_1,u_2\in W_{k-1}$ and $w=u_1^Lu_2^R$. The converse and the uniqueness are clear. Hence $|W_k|=2|W_{k-1}|^2$, so $|W_k|=2^{2^k-1}$, which is the $2$-part of $(2^k)!$.

(b) is clear. (c) For $i\le j$ the residue class of $\beta+2^j\ell$ modulo $2^i$ is determined by $\beta$, and for $i>j$ it is determined by $\beta$ and by $\ell$ modulo $2^{i-j}$. This shows that $u_{(\beta)},\hat r\in W_k$, and that an element of $W_k$ fixing every block $B_\beta$ induces on each $B_\beta$ an element of $W_a$. Since $B\cap\hat W_j=1$ and $|B||\hat W_j|=2^{2^j(2^a-1)}\cdot2^{2^j-1}=|W_k|$, the remaining assertions follow.
\end{proof}

For $X\subseteq\Omega_k$ we define $\sym(X)\ge0$ recursively by $\sym(X)=0$ if $k=0$ and
\[\sym(X)=\sym(X_L)+\sym(X_R)+\epsilon(X)\qquad(k\ge1),\]
where $\epsilon(X)=1$ if $X_L$ and $X_R$ are $W_{k-1}$-conjugate, and $\epsilon(X)=0$ otherwise.

\begin{lemma}\label{lem:subsets}
Let $k\ge1$ and $X,X'\subseteq\Omega_k$. Then $X$ and $X'$ are $W_k$-conjugate if and only if $X_L,X_R$ are $W_{k-1}$-conjugate to $X'_L,X'_R$ or to $X'_R,X'_L$, respectively. Moreover, $\{u_1^Lu_2^R:u_1\in\Stab(X_L),\,u_2\in\Stab(X_R)\}$ is equal to $\Stab_{W_k}(X)$ if $\epsilon(X)=0$, and has index $2$ in it if $\epsilon(X)=1$. In particular $|\Stab_{W_k}(X)|=2^{\sym(X)}$.
\end{lemma}

\begin{proof}
This follows from Lemma~\ref{lem:tree}(a), because $(X^{u_1^Lu_2^R})_L=X_L^{u_1}$ and $(X^{u_1^Lu_2^R})_R=X_R^{u_2}$, while $(X^\tau)_L=X_R$ and $(X^\tau)_R=X_L$.
\end{proof}

\subsection{The monomial groups}
Let $n\ge1$ and let $\K$ be a finite field of order $2^n+1$. By Lemma~\ref{lem:arith} below, either $\K=\F_q$ for the Fermat prime $q=2^n+1$, or $n=3$ and $\K=\F_9$. Let $p$ be the characteristic of $\K$, let $A=\K^\times$, a cyclic group of order $2^n$, and fix a generator $\zeta$ of $A$.

For $k\ge0$ let $V=\K^{\Omega_k}$ be the space of row vectors $v=(v_\omega)_{\omega\in\Omega_k}$, with standard basis $(e_\omega)$. Let $H$ be the group of diagonal matrices $d=\diag(d_\omega)$ with $d_\omega\in A$, and put $\det d=\prod_\omega d_\omega$. We let $W_k$ act on $V$ by permutation matrices, $e_\omega w=e_{\omega^w}$, and put $P(n,k)=HW_k=H\rtimes W_k$, the group of monomial matrices with nonzero entries in $A$ and underlying permutation in $W_k$. We regard $V$ as an $\F_pP(n,k)$-module. For $X\subseteq\Omega_k$ let $A^X=\{d\in H: d_\omega=1\text{ for }\omega\notin X\}$ and $v_X=\sum_{\omega\notin X}e_\omega$, and for $v\in V$ let $Z(v)=\{\omega:v_\omega=0\}$ be its \emph{zero set}.

\begin{lemma}\label{lem:orbits}
Let $P=P(n,k)$.
\begin{enumerate}[(a)]
\item Two vectors $v,v'\in V$ lie in the same $P$-orbit if and only if $Z(v)$ and $Z(v')$ are $W_k$-conjugate.
\item For $X\subseteq\Omega_k$ we have $C_P(v_X)=A^X\rtimes\Stab_{W_k}(X)$, and every $v$ with $Z(v)=X$ is $H$-conjugate to $v_X$.
\item $|v^P|=|P|\cdot2^{-n|Z(v)|-\sym(Z(v))}$ for every $v\in V$.
\end{enumerate}
\end{lemma}

\begin{proof}
Since $A$ acts regularly on $\K\setminus\{0\}$, the group $H$ acts transitively on the vectors with a given zero set, and an element $dw$ ($d\in H$, $w\in W_k$) maps a vector with zero set $X$ to one with zero set $X^w$. This proves (a). An element $dw$ fixes $v_X$ if and only if $w$ stabilises $X$ and $d_\omega=1$ for $\omega\notin X$, which gives (b), and (c) follows from (b) and Lemma~\ref{lem:subsets}.
\end{proof}

If $\K=\F_9$, then for $\omega\in\Omega_k$ let $\varphi_\omega$ be the map applying the field automorphism $x\mapsto x^3$ to the $\omega$-coordinate, and let $\tilde P(3,k)=\langle P(3,k),\varphi_\omega:\omega\in\Omega_k\rangle$; this is the group $\GaL_1(9)\wr W_k$ acting on $V=\F_9^{2^k}$.

\begin{lemma}[cf.\ \cite{We,CF}]\label{lem:sylow}
\begin{enumerate}[(a)]
\item If $\K=\F_q$, then $P(n,k)$ is a Sylow $2$-subgroup of $\GL_{2^k}(q)$.
\item If $\K=\F_9$, then $\tilde P(3,k)$ is a Sylow $2$-subgroup of $\GL_{2^{k+1}}(3)$, and $\tilde P(3,k)$ and $P(3,k)$ have the same orbits on $V$.
\end{enumerate}
\end{lemma}

\begin{proof}
For an integer $N$ let $\nu(N)$ be the exponent of $2$ in $N$. (a) As $q\equiv1\pmod4$, induction on $\nu(i)$ gives $\nu(q^i-1)=n+\nu(i)$: for odd $i$ the number $(q^i-1)/(q-1)=1+q+\dots+q^{i-1}$ is odd, and $q^{2i}-1=(q^i-1)(q^i+1)$ with $q^i+1\equiv2\pmod4$. With $d=2^k$ we obtain $\nu(|\GL_d(q)|)=\sum_{i=1}^d(n+\nu(i))=nd+\nu(d!)=nd+d-1$, which is $\log_2|P(n,k)|$. (b) Similarly, $\nu(3^i-1)=1$ for odd $i$ and $\nu(3^i-1)=2+\nu(i)$ for even $i$. With $d=2^{k+1}$ this gives $\nu(|\GL_d(3)|)=d/2+d+\nu(d!)=5\cdot2^k-1=\log_2|\tilde P(3,k)|$, and $\tilde P(3,k)\le\GL_d(3)$. Finally, $\tilde P(3,k)$ maps vectors with zero set $X$ to vectors with zero sets in the $W_k$-class of $X$, and $P(3,k)$ is already transitive on these vectors by Lemma~\ref{lem:orbits}.
\end{proof}

Accordingly, for a Fermat prime $q\ge5$ we put $P_m=P(n,m-1)$, and for $q=3$ and $m\ge2$ we put $P_m=\tilde P(3,m-2)$; these are the groups of the introduction. The largest orbit size and the largest orbits of $P_m$ are those of $P(n,k)$, where $k$ is the level.

\subsection{Three general lemmas}

\begin{lemma}\label{lem:criterion}
Let $P=P(n,k)$, let $2^x$ be the largest orbit size of $P$ on $V$, and let $T\le P$ with $|T|=2^x$. Then every regular orbit of $T$ is a largest orbit of $P$, and a largest orbit $\cO$ of $P$ is a regular orbit of $T$ if and only if $T\cap C_P(v)=1$ for some (equivalently, every) $v\in\cO$.
\end{lemma}

\begin{proof}
This is \cite[Lemma~4.4]{KR}: a regular orbit of $T$ has size $2^x$ and lies in an orbit of $P$; and for $v\in\cO$ we have $v^T=\cO$ if and only if $|T:T\cap C_P(v)|=2^x$, in which case all point stabilisers in $T$ are conjugate.
\end{proof}

\begin{lemma}\label{lem:irr}
Let $T\le P(n,k)$, let $M=T\cap H$, and let $\bar T$ be the image of $T$ in $W_k$. Assume that
\begin{enumerate}[(i)]
\item $\bar T$ is transitive on $\Omega_k$;
\item for every $\omega$ there is $d\in M$ with $\F_p(d_\omega)=\K$; and
\item for any two distinct $\omega,\omega'$ there is $d\in M$ with $d_\omega\ne1=d_{\omega'}$.
\end{enumerate}
Then $V$ is an irreducible $\F_pT$-module. If (i) fails, then $V$ is reducible.
\end{lemma}

\begin{proof}
The lines $L_\omega=\K e_\omega$ are $\F_pM$-submodules, and $d\in M$ acts on $L_\omega$ as multiplication by $d_\omega$. By (ii), $L_\omega$ is an irreducible $\F_pM$-module, and by (iii), $L_\omega\not\cong L_{\omega'}$ for $\omega\ne\omega'$. As $M$ is a $2$-group and $p$ is odd, $V=\bigoplus_\omega L_\omega$ is a direct sum of pairwise non-isomorphic irreducible modules, so every $\F_pM$-submodule of $V$ is a sum of some of the $L_\omega$. Since $T$ permutes the $L_\omega$ as $\bar T$ permutes $\Omega_k$, (i) implies that $V$ is irreducible. If $\Lambda$ is a $\bar T$-orbit, then $\bigoplus_{\omega\in\Lambda}L_\omega$ is $T$-invariant.
\end{proof}

If $\K=\F_q$, condition (ii) is automatic. In our applications, (ii) and (iii) will be checked with the elements $d^{\omega,\omega''}\in H$ which have entry $\zeta$ at $\omega$, entry $\zeta^{-1}$ at $\omega''\ne\omega$, and entry $1$ elsewhere; note that $\det d^{\omega,\omega''}=1$.

\begin{lemma}\label{lem:count}
Let $M\le H$ and $X\subseteq\Omega_k$ with $|M|\,|A|^{|X|}>|H|$. Then $M\cap A^X\ne1$. Consequently, if $T\le P(n,k)$ and $T\cap H\supseteq M$, then the orbit of $v_X$ is not a regular orbit of $T$.
\end{lemma}

\begin{proof}
The first assertion holds as $|MA^X|\le|H|$, and the second follows since $A^X\le C_{P(n,k)}(v_X)$.
\end{proof}

\section{Reduction: proof of Theorem~\ref{thmA}}\label{sec:A}

\begin{lemma}\label{lem:dim}
Let $G$ be a $2$-group and $V$ an irreducible $\F_qG$-module, where $q$ is an odd prime. Then $\dim V$ is a power of $2$.
\end{lemma}

\begin{proof}
We argue by induction on $|G|$. If $G=1$, then $\dim V=1$. Otherwise let $M$ be a subgroup of index $2$, let $W$ be an irreducible $\F_qM$-submodule of $V$, and let $g\in G\setminus M$. Then $Wg$ is an $M$-submodule, $W+Wg$ is $G$-invariant and hence equal to $V$, and $W\cap Wg$ is $0$ or $W$. Thus $\dim V\in\{\dim W,2\dim W\}$, and $\dim W$ is a power of $2$ by induction.
\end{proof}

\begin{lemma}\label{lem:arith}
Let $r$ be a prime and $e\ge1$ such that $r^e-1$ is a power of $2$. Then $r^e=2$, or $e=1$ and $r$ is a Fermat prime, or $r^e=9$.
\end{lemma}

\begin{proof}
If $r=2$, then $e=1$. Let $r$ be odd. If $e$ is even, then $(r-1)(r+1)$ divides $r^e-1$, so $r-1$ and $r+1$ are powers of $2$ and $r=3$; by the proof of Lemma~\ref{lem:sylow}, $3^e-1=2^{2+\nu(e)}\le4e$, so $e=2$. If $e\ge3$ is odd, then $(r^e-1)/(r-1)$ is odd and greater than $1$, which is impossible. Finally, if $2^a+1$ is prime, then $a$ is a power of $2$.
\end{proof}

\begin{proposition}\label{prop:centraliser}
Let $G$, $F$ and $V$ be as in Theorem~\ref{thmA}, let $q=\operatorname{char}F$ and $C=C_{\GL(V)}(G)$.
\begin{enumerate}[(a)]
\item $C$ is isomorphic to a subgroup of $G$.
\item $q$ is a Fermat prime, $V$ is irreducible as an $\F_qG$-module, and $\End_{\F_qG}(V)$ is $\F_q$ or $\F_9$.
\end{enumerate}
\end{proposition}

\begin{proof}
(a) Let $\cO$ be the regular orbit and $c\in C$. Since $C_G(vc)=C_G(v)$ for all $v\in V$, the map $c$ permutes the regular orbits, so $\cO c=\cO$. Fix $v\in\cO$. For $c\in C$ there is a unique $g_c\in G$ with $vc=vg_c$, and as $c$ commutes with $G$, we have $g_{cc'}=g_{c'}g_c$; so $c\mapsto g_c^{-1}$ is a homomorphism. If $g_c=1$, then $c$ fixes $vh$ for every $h\in G$, i.e., $c$ fixes $\cO$ pointwise. The set $\cO$ is invariant under the scalars $F^\times\le C$, so its $\F_q$-span is an $FG$-submodule, hence equal to $V$, and $c=1$.

(b) If $q=2$, then $V$ is the trivial module, contradicting $G\ne1$. By (a), $F^\times$ is a $2$-group, so $q$ is a Fermat prime by Lemma~\ref{lem:arith}. Let $W$ be an irreducible $\F_qG$-submodule of $V$. Then $\sum_{\lambda\in F^\times}\lambda W$ is a nonzero $FG$-submodule, so $V=W_1\oplus\dots\oplus W_e$ with $W_i\cong W$. If $e\ge2$, fix an isomorphism $\theta\colon W_1\to W_2$; then $w_1+w_2+\dots+w_e\mapsto w_1+(w_2+\theta(w_1))+w_3+\dots+w_e$ ($w_i\in W_i$) is an element of $C$ of order $q$, contradicting (a). Hence $V$ is irreducible over $\F_q$, and $K=\End_{\F_qG}(V)$ is a finite field with $K^\times\le C$. By (a) and Lemma~\ref{lem:arith}, $K$ is $\F_q$ or $\F_9$.
\end{proof}

\begin{proposition}\label{prop:primitive}
Let $G$ be a $2$-group, $q$ a Fermat prime and $V$ a faithful irreducible $\F_qG$-module. Let $V=V_1\oplus\dots\oplus V_r$ be a decomposition into subspaces permuted transitively by $G$ with $\dim V_1$ minimal, let $H$ be the stabiliser of $V_1$ and $\bar H=H/C_H(V_1)$. Then $V\cong\Ind_H^G(V_1)$, and either $\dim V_1=1$, or $q=3$, $\dim V_1=2$ and $\bar H\in\{C_8,Q_8,\SD_{16}\}$.
\end{proposition}

\begin{proof}
By \cite[\S1]{MW}, $V\cong\Ind_H^G(V_1)$ and $V_1$ is a primitive faithful irreducible $\F_q\bar H$-module; in particular every normal subgroup of $\bar H$ acts homogeneously on $V_1$ (Clifford). A normal abelian subgroup $N$ of $\bar H$ therefore embeds into the multiplicative group of the field $\End_{\F_qN}(X)$, $X$ an irreducible constituent of $V_1|_N$, so $N$ is cyclic. Hence $\bar H$ is cyclic, generalised quaternion of order at least $8$, or dihedral or semidihedral of order at least $16$ \cite[Theorem~5.4.10]{Go}.

Suppose that $\bar H$ is not cyclic, let $C=\langle c\rangle$ be a cyclic subgroup of index $2$, of order $2^{a-1}\ge4$, and let $c^t=c^\varepsilon$ for $t\in\bar H\setminus C$, where $\varepsilon=-1$ or $\varepsilon=2^{a-2}-1$. As $V_1|_C$ is homogeneous and faithful, the eigenvalues of $c$ on $V_1\otimes\overline{\F}_q$ are the conjugates $\lambda^{q^i}$ of one primitive $2^{a-1}$-th root of unity $\lambda$. Since $c$ and $c^\varepsilon$ are conjugate, $\varepsilon\equiv q^i\pmod{2^{a-1}}$ for some $i$. If $q\ge5$, then $q\equiv1\pmod4$ while $\varepsilon\equiv3\pmod4$, a contradiction. If $q=3$, then the powers of $3$ are $\equiv1$ or $3\pmod8$, while $\varepsilon\equiv7\pmod8$ unless $\bar H\cong Q_8$ or $\bar H\cong\SD_{16}$; the faithful irreducible $\F_3$-modules of these two groups have dimension $2$.

Now let $\bar H=\langle c\rangle$ be cyclic of order $2^b$, so that $V_1\cong\F_q(\lambda)$ with $c$ acting as multiplication by a primitive $2^b$-th root of unity $\lambda$. If $\lambda\notin\F_q(\lambda^2)$, then $V_1=\F_q(\lambda^2)\oplus\lambda\F_q(\lambda^2)$ is a decomposition into two subspaces interchanged by $c$, contradicting primitivity. So $\F_q(\lambda)=\F_q(\lambda^2)$. If $q\ge5$, then the multiplicative order of $q$ modulo $2^b$ is $1$ for $b\le n$ and $2^{b-n}$ for $b>n$, so $\F_q(\lambda)=\F_q(\lambda^2)$ forces $b\le n$ and $\dim V_1=1$. If $q=3$, the same reasoning leaves $b=1$ and $\dim V_1=1$, or $b=3$ and $V_1\cong\F_9$.
\end{proof}

\begin{proof}[Proof of Theorem~\ref{thmA}]
Parts (a) and (b) are Proposition~\ref{prop:centraliser}. By (b) and Lemma~\ref{lem:dim}, $\dim_{\F_q}V=2^{m-1}$ for some $m\ge1$, and by Sylow's theorem and Lemma~\ref{lem:sylow}, $G$ is conjugate in $\GL(V)\cong\GL(V_m)$ to a subgroup of $P_m$. This proves (c), and (d) is Proposition~\ref{prop:primitive}.
\end{proof}

\section{The largest orbits: proof of Theorem~\ref{thmB}}\label{sec:B}

Throughout this section $n\ge1$, $k\ge0$ and $P=P(n,k)$. For $X\subseteq\Omega_k$ we call
\[c(X)=n|X|+\sym(X)\]
the \emph{cost} of $X$. By Lemma~\ref{lem:orbits}, $|v^P|=|P|\cdot2^{-c(Z(v))}$, so the largest orbits of $P$ correspond to the $W_k$-classes of subsets of minimal cost. We write $a_k$ for this minimal cost. Since $W_k$ is transitive, all one-element subsets of $\Omega_k$ are conjugate.

\begin{definition}\label{def:Y}
For $j\ge2$ define subsets $Y_j^-$, $Y_j^0$ and $Y_j^+$ of $\Omega_j$ by $Y_2^-=\{0\}$, $Y_2^0=\{0,1\}$, $Y_2^+=\{0,1,3\}$ and
\[Y_{j+1}^-=(Y_j^-,Y_j^0),\qquad Y_{j+1}^0=(Y_j^-,Y_j^+),\qquad Y_{j+1}^+=(Y_j^0,Y_j^+).\]
By induction $|Y_j^\delta|=2^{j-1}+\delta$, where $\delta\in\{-,0,+\}$ is read as $-1,0,1$, and $Y_j^\delta\subseteq\{0,1,\dots,2^j-1\}$. We regard $Y_j^\delta$ also as a subset of $\Omega_k$ for all $k\ge j$.
\end{definition}

Thus $Y_2^-$ is a point, $Y_2^0$ consists of one point in each half of $\Omega_2$, and $Y_2^+$ consists of one point in one half and both points of the other half. Viewed in $\Omega_k$, the set $Y_2^+=\{0,1,3\}$ consists of one point in one half and two points in different quarters of the other half.

\begin{lemma}\label{lem:costrec}
For $k\ge1$ and $X\subseteq\Omega_k$ we have $c(X)=c(X_L)+c(X_R)+\epsilon(X)$.
\end{lemma}

\begin{proof}
This is immediate from $|X|=|X_L|+|X_R|$ and the definition of $\sym$.
\end{proof}

\begin{lemma}\label{lem:costlow}
Let $0\le k\le n$. Then $c(\emptyset)=2^k-1$, $c(\{0\})=2^k-1+n-k$, and $c(X)\ge2^k+1+n-k$ for every $X\subseteq\Omega_k$ with $|X|\ge2$.
\end{lemma}

\begin{proof}
For $k=0$ this is clear. Let $1\le k\le n$ and assume the assertion for $k-1$. By Lemma~\ref{lem:costrec}, $c(\emptyset)=2(2^{k-1}-1)+1=2^k-1$ and, for a point, $c=(2^{k-1}-1)+(2^{k-1}+n-k)=2^k-1+n-k$. Let $|X|\ge2$. If $X_L$ and $X_R$ are points, then $c(X)=2(2^{k-1}+n-k)+1\ge2^k+1+n-k$. Otherwise one of $X_L,X_R$ has at least two elements, and $c(X)\ge(2^{k-1}-1)+(2^{k-1}+2+n-k)=2^k+1+n-k$.
\end{proof}

\begin{lemma}\label{lem:costhigh}
\begin{enumerate}[(a)]
\item Let $k=n+1$. Then $c(\{0\})=2^{n+1}-2$, $c(\emptyset)=c(\{0,1\})=2^{n+1}-1$, and $c(X)\ge2^{n+1}$ for every $X$ that is not conjugate to $\emptyset$, $\{0\}$ or $\{0,1\}$.
\item Let $k=n+2$. Then $c(Y_2^\delta)=2^{n+2}-3$ for $\delta\in\{-,0,+\}$, and $c(X)\ge2^{n+2}-2$ for every $X$ that is not conjugate to one of these three sets.
\item Let $k\ge n+2$ and $j=k-n$. Then $a_k=2^{j-2}(2^{n+2}-3)$, and the subsets of cost $a_k$ are exactly the conjugates of $Y_j^-$, $Y_j^0$ and $Y_j^+$, which are pairwise non-conjugate.
\end{enumerate}
\end{lemma}

\begin{proof}
(a) By Lemma~\ref{lem:costlow} for $k=n$, the empty set and the points of $\Omega_n$ have cost $2^n-1$, and all other subsets have cost at least $2^n+1$. By Lemma~\ref{lem:costrec}, $X=(X_L,X_R)$ has cost $2^{n+1}-2$ if one of $X_L,X_R$ is empty and the other a point, cost $2^{n+1}-1$ if both are empty or both are points, and cost at least $2^{n+1}$ otherwise. By Lemma~\ref{lem:subsets} these cases correspond to $X$ conjugate to $\{0\}$, to $\emptyset$ or $\{0,1\}=(\{0\},\{0\})$, and to the remaining sets.

(b) Apply Lemma~\ref{lem:costrec} with the values in (a). If both halves are conjugate to $\{0\}$, the cost is $2(2^{n+1}-2)+1=2^{n+2}-3$; if one half is conjugate to $\{0\}$ and the other to $\emptyset$ or $\{0,1\}$, the cost is $(2^{n+1}-2)+(2^{n+1}-1)=2^{n+2}-3$. These three cases give the conjugates of $(\{0\},\{0\})=Y_2^0$, $(\{0\},\emptyset)=Y_2^-$ and $(\{0\},\{0,1\})=Y_2^+$. In all other cases the cost is at least $2^{n+2}-2$.

(c) For $j=2$ this is (b), since the three sets have different sizes. Let $j\ge3$ and assume (c) for $k-1$. By Lemma~\ref{lem:costrec}, $c(X)\ge c(X_L)+c(X_R)\ge2a_{k-1}$, with equality if and only if $X_L$ and $X_R$ have cost $a_{k-1}$ and are not conjugate. By induction and Lemma~\ref{lem:subsets}, this means that $X$ is conjugate to $(Y_{j-1}^\alpha,Y_{j-1}^\beta)$ for some $(\alpha,\beta)\in\{(-,0),(-,+),(0,+)\}$, that is, to $Y_j^-$, $Y_j^0$ or $Y_j^+$. These have different sizes, and $a_k=2a_{k-1}$.
\end{proof}

\begin{theorem}\label{thm:largest}
Let $n\ge1$ and $k\ge0$. Then
\[a_k=\begin{cases}2^k-1,& k\le n,\\ 2^{n+1}-2,& k=n+1,\\ 2^{k-n-2}(2^{n+2}-3),& k\ge n+2,\end{cases}\]
and the largest orbit size of $P(n,k)$ on $V$ is $2^x$ with $x=\log_2|P(n,k)|-a_k=n2^k+2^k-1-a_k$. Let $\cO_\emptyset$ be the set of vectors without zero coordinates and $\cO_1$ the set of vectors with exactly one zero coordinate. The largest orbits are $\cO_\emptyset$ if $k<n$; $\cO_\emptyset$ and $\cO_1$ if $k=n$; $\cO_1$ if $k=n+1$; and, if $k\ge n+2$, the three orbits $\cO^\delta=\cO^\delta_k$ of the vectors $v_{Y^\delta_{k-n}}$, $\delta\in\{-,0,+\}$, whose elements have exactly $2^{k-n-1}+\delta$ zero coordinates.
\end{theorem}

\begin{proof}
This follows from Lemmas~\ref{lem:orbits}, \ref{lem:costlow} and~\ref{lem:costhigh}.
\end{proof}

Note that $\cO^-_{n+2}=\cO_1$, since $Y_2^-=\{0\}$.

\begin{proof}[Proof of Theorem~\ref{thmB}]
By Lemma~\ref{lem:sylow}, the orbits of $P_m$ on $V_m$ are those of $P(n,k)$. Evaluating $x$ in Theorem~\ref{thm:largest} gives the formula for $x_m$, and the number of largest orbits can be read off. If $k\ge n+2$, then $x_{m+1}=2x_m+1$ since $a_{k+1}=2a_k$.
\end{proof}

\section{Existence: proof of Theorem~\ref{thmC}}\label{sec:C}

In this section $n\ge2$ and $\K$ is a field of order $2^n+1$; this covers all cases of interest, namely $\K=\F_q$ for a Fermat prime $q\ge5$, and $\K=\F_9$ with $n=3$. Let $x_k$ denote the exponent of the largest orbit size of $P(n,k)$. Let $\zeta_0\in H$ be the diagonal matrix with entry $\zeta$ at $0$ and $1$ elsewhere, let $g_k$ also denote the permutation matrix of the $2^k$-cycle $g_k$, and let $M_k=\{d\in H:\det d=1\}$. As $\det$ is invariant under conjugation by $W_k$ and $H$ is abelian, every subgroup of $H$ defined by a condition on $\det$ is normal in $P(n,k)$.

\subsection{Small levels}

\begin{proposition}\label{prop:low}
Let $0\le k<n$, put $N_k=\{d\in H:\det d\in\langle\zeta^{2^k}\rangle\}$ and $T=\langle N_k,\zeta_0g_k\rangle$. Then $|T|=|H|=2^{x_k}$, and $T$ is an irreducible subgroup of $P(n,k)$ whose unique regular orbit is $\cO_\emptyset$.
\end{proposition}

\begin{proof}
As $N_k$ is normal, $T=N_k\langle\zeta_0g_k\rangle$. An element $d(\zeta_0g_k)^i$ with $d\in N_k$ lies in $H$ if and only if $2^k$ divides $i$, and $(\zeta_0g_k)^{2^k}$ is the scalar matrix $\zeta$, whose determinant $\zeta^{2^k}$ lies in $\langle\zeta^{2^k}\rangle$. Hence $T\cap H=N_k$ and $|T|=2^k|N_k|=|H|$, because $\det\colon H\to A$ is surjective and $|A:\langle\zeta^{2^k}\rangle|=2^k$. By Theorem~\ref{thm:largest}, $|H|=2^{x_k}$ and $\cO_\emptyset$ is the only largest orbit.

Let $t=d(\zeta_0g_k)^i\in T\cap C(v_\emptyset)=T\cap W_k$ with $d\in N_k$ and $0\le i<2^k$. The diagonal part of $(\zeta_0g_k)^i$ has determinant $\zeta^i$, so $\det(d)\zeta^i=1$ and $\zeta^i\in\langle\zeta^{2^k}\rangle$. As $\zeta$ has order $2^n$, this forces $2^k\mid i$, so $i=0$ and $t=d\in H\cap W_k=1$. By Lemma~\ref{lem:criterion}, $\cO_\emptyset$ is the unique regular orbit of $T$.

For irreducibility we use Lemma~\ref{lem:irr}; condition (i) holds as $g_k$ is a $2^k$-cycle. If $k=0$, then $T=A$ acts irreducibly on $V=\K$. If $k=1$, then $\diag(\zeta^2,1),\diag(1,\zeta^2)\in N_1$, and $\zeta^2\ne1$ generates $\K$ over $\F_p$ (if $\K=\F_9$, then $\zeta^2$ has order $4$). If $k\ge2$ and $\omega\ne\omega'$, choose $\omega''\notin\{\omega,\omega'\}$; then $d^{\omega,\omega''}\in M_k\le N_k$ has the required properties.
\end{proof}

\begin{proposition}\label{prop:mid}
Let $n\le k\le n+2$ and $\Gamma_k=M_k\rtimes\langle g_k\rangle$. Then $|\Gamma_k|=2^{x_k}$, and $\Gamma_k$ is an irreducible subgroup of $P(n,k)$ whose unique regular orbit is $\cO_1$, the set of vectors with exactly one zero coordinate.
\end{proposition}

\begin{proof}
We have $|\Gamma_k|=2^k|M_k|=2^{k-n}|H|=2^{n2^k+k-n}$, and $n2^k+k-n=x_k$ for $n\le k\le n+2$ by Theorem~\ref{thm:largest}. If $dg_k^i\in\Gamma_k\cap C(v_{\{0\}})=A^{\{0\}}\rtimes\Stab_{W_k}(\{0\})$, then $g_k^i$ fixes $0$, so $i=0$ and $d\in M_k\cap A^{\{0\}}=1$. Hence $\cO_1$ is a regular orbit of $\Gamma_k$. If $k=n$, the other largest orbit $\cO_\emptyset$ is not regular, because $1\ne g_n\in\Gamma_n\cap C(v_\emptyset)$. If $k=n+1$, then $\cO_1$ is the only largest orbit. If $k=n+2$, then $\cO_1=\cO^-$, and the other two largest orbits are the orbits of $v_X$ with $|X|\in\{2,3\}$; since $|M_k|\,|A|^{|X|}>|H|$, they are not regular by Lemma~\ref{lem:count}. Irreducibility follows from Lemma~\ref{lem:irr} with the elements $d^{\omega,\omega''}\in M_k$, as $|\Omega_k|\ge4$.
\end{proof}

\subsection{A Singer-type group and a reduction}

Let $\Sigma=\langle M_n,\zeta_0g_n\rangle\le P(n,n)$. The element $\zeta_0g_n$ has order $2^{2n}$, and its $2^n$-th power is the scalar matrix $\zeta$; the letter $\Sigma$ is meant to recall a Singer cycle.

\begin{lemma}\label{lem:Sigma}
We have $\Sigma\cap H=M_n$ and $|\Sigma|=|H|$, and $\Sigma$ is regular on both largest orbits $\cO_\emptyset$ and $\cO_1$ of $P(n,n)$.
\end{lemma}

\begin{proof}
As in the proof of Proposition~\ref{prop:low} (with $k=n$), $\Sigma\cap H=M_n$ and $|\Sigma|=2^n|M_n|=|H|$, which is the size of $\cO_\emptyset$ and of $\cO_1$ by Theorem~\ref{thm:largest}. If $d(\zeta_0g_n)^i\in W_n$ with $d\in M_n$ and $0\le i<2^n$, then $\zeta^i=1$, so $i=0$ and the element is trivial; thus $\Sigma\cap C(v_\emptyset)=1$. If $d(\zeta_0g_n)^i$ fixes $v_{\{0\}}$, then $g_n^i$ fixes $0$, so $i=0$ and $d\in M_n\cap A^{\{0\}}=1$.
\end{proof}

Now let $j\ge2$ and $k=n+j$, and let $\Omega_k=\bigcup_{\beta\in\Omega_j}B_\beta$ with $B_\beta=\beta+2^j\Omega_n$, as in Lemma~\ref{lem:tree}(c) with $a=n$. For $v\in V=\K^{\Omega_k}$ and $\beta\in\Omega_j$ let $v_\beta\in\K^{\Omega_n}$ be given by $(v_\beta)_\ell=v_{\beta+2^j\ell}$. For $\sigma\in P(n,n)$ let $\sigma_{(\beta)}\in P(n,k)$ act on the coordinates in $B_\beta$ as $\sigma$ acts on $\K^{\Omega_n}$ (via $\ell\mapsto\beta+2^j\ell$) and trivially on the other coordinates; and for $r\in W_j$ let $\hat r\in W_k$ be as in Lemma~\ref{lem:tree}(c). Then $(v\sigma_{(\beta)})_\beta=v_\beta\sigma$ and $(v\hat r)_{\beta^r}=v_\beta$. For $R\le W_j$ put
\[T(R)=\Bigl(\prod_{\beta\in\Omega_j}\Sigma_{(\beta)}\Bigr)\rtimes\hat R,\qquad\text{where }\Sigma_{(\beta)}=\{\sigma_{(\beta)}:\sigma\in\Sigma\},\ \hat R=\{\hat r:r\in R\}.\]

\begin{proposition}\label{prop:reduction}
Let $j\ge2$, $k=n+j$ and $R\le W_j$.
\begin{enumerate}[(a)]
\item $T(R)$ is a subgroup of $P(n,k)$ of order $|H|\cdot|R|$, and $|T(R)|=2^{x_k}$ if and only if $|R|=2^{3\cdot2^{j-2}-1}$.
\item Let $\delta\in\{-,0,+\}$ and $v=v_{Y_j^\delta}$. Then $T(R)\cap C_{P(n,k)}(v)$ is isomorphic to $\Stab_R(Y_j^\delta)$.
\item If $|R|=2^{3\cdot2^{j-2}-1}$, then the regular orbits of $T(R)$ are exactly the orbits $\cO^\delta_k$ with $\Stab_R(Y^\delta_j)=1$.
\item If $R$ is transitive on $\Omega_j$, then $T(R)$ is irreducible; otherwise it is reducible.
\end{enumerate}
\end{proposition}

\begin{proof}
(a) By Lemma~\ref{lem:tree}(c), $\hat r^{-1}\sigma_{(\beta)}\hat r=\sigma_{(\beta^r)}$, so $\hat R$ normalises $\prod_\beta\Sigma_{(\beta)}$, and the two subgroups intersect trivially. Hence $|T(R)|=|\Sigma|^{2^j}|R|=|H||R|$ by Lemma~\ref{lem:Sigma}. By Theorem~\ref{thm:largest}, $x_k-\log_2|H|=3\cdot2^{j-2}-1$.

(b) Let $v=v_{Y_j^\delta}$. As $Y_j^\delta\subseteq\{0,\dots,2^j-1\}$, each block $B_\beta$ contains at most one point of $Y_j^\delta$, namely $\beta$, and it does so if and only if $\beta\in Y_j^\delta$. Hence $v_\beta=v_{\{0\}}\in\cO_1$ if $\beta\in Y_j^\delta$ and $v_\beta=v_\emptyset\in\cO_\emptyset$ otherwise. An element $t=\prod_\beta(\sigma_\beta)_{(\beta)}\,\hat r$ of $T(R)$ satisfies $(vt)_{\beta^r}=v_\beta\sigma_\beta$, so $t$ fixes $v$ if and only if $v_\beta\sigma_\beta=v_{\beta^r}$ for all $\beta$. As $\sigma_\beta$ preserves the number of zero coordinates, this forces $r\in\Stab_R(Y^\delta_j)$. Conversely, for $r\in\Stab_R(Y_j^\delta)$ and each $\beta$ the vectors $v_\beta$ and $v_{\beta^r}$ lie in the same orbit $\cO_\emptyset$ or $\cO_1$, so by Lemma~\ref{lem:Sigma} there is exactly one $\sigma_\beta\in\Sigma$ with $v_\beta\sigma_\beta=v_{\beta^r}$. Hence $t\mapsto r$ is an isomorphism $T(R)\cap C(v)\to\Stab_R(Y_j^\delta)$.

(c) follows from (a), (b), Lemma~\ref{lem:criterion} and Theorem~\ref{thm:largest}.

(d) The image of $T(R)$ in $W_k$ contains $(g_n)_{(\beta)}$ for every $\beta$, and $\hat R$; so it is transitive if and only if $R$ is, and if $R$ is not transitive, $T(R)$ is reducible by Lemma~\ref{lem:irr}. Moreover $T(R)\cap H$ contains $(M_n)_{(\beta)}$ for every $\beta$. Given $\omega\ne\omega'$ with $\omega\in B_\beta$, choose $\omega''\in B_\beta\setminus\{\omega,\omega'\}$, which is possible as $|B_\beta|=2^n\ge4$. Then $d^{\omega,\omega''}\in(M_n)_{(\beta)}$ satisfies conditions (ii) and (iii) of Lemma~\ref{lem:irr}, and $T(R)$ is irreducible if $R$ is transitive.
\end{proof}

\subsection{The permutation problem}

By Proposition~\ref{prop:reduction}, it remains to find, for each $j\ge3$, a transitive subgroup $R\le W_j$ of order $2^{3\cdot2^{j-2}-1}$ with $\Stab_R(Y_j^\delta)=1$ for exactly one $\delta$. Throughout, $W_j$ acts on the subsets of $\Omega_j$. For $j\ge1$ and $D\le\{u_1^Lu_2^R:u_1,u_2\in W_j\}$ we write $D\langle\tau\rangle$ for the subgroup of $W_{j+1}$ generated by $D$ and $\tau=\tau_{j+1}$, whenever $\tau$ normalises $D$.

\begin{lemma}\label{lem:halves}
Let $j\ge2$, $\delta\in\{-,0,+\}$, and write $Y_{j+1}^\delta=(Y_1,Y_2)$. Let $D\le\{u_1^Lu_2^R:u_1,u_2\in W_j\}$ be normalised by $\tau$. Then $\Stab_{D\langle\tau\rangle}(Y^\delta_{j+1})=\{u_1^Lu_2^R\in D:u_1\in\Stab(Y_1),\,u_2\in\Stab(Y_2)\}$.
\end{lemma}

\begin{proof}
By Definition~\ref{def:Y}, $|Y_1|\ne|Y_2|$, so no element interchanging the halves stabilises $Y^\delta_{j+1}$, and the elements of $D\langle\tau\rangle\setminus D$ interchange the halves.
\end{proof}

\begin{lemma}\label{lem:level2}
In $W_2=\langle\tau,(0\,2),(1\,3)\rangle\cong D_8$, where $\tau=(0\,1)(2\,3)$, consider $C_4=\langle g_2\rangle=\langle(0\,1\,2\,3)\rangle$, $E_2=\{1,(0\,1)(2\,3),(0\,2)(1\,3),(0\,3)(1\,2)\}$ and $B_2=\langle(0\,2),(1\,3)\rangle$.
\begin{enumerate}[(a)]
\item $\Stab_{C_4}(Y_2^\delta)=1$ for all $\delta$.
\item $\Stab_{E_2}(Y_2^-)=\Stab_{E_2}(Y_2^+)=1\ne\Stab_{E_2}(Y_2^0)$.
\item $\Stab_{B_2}(Y_2^0)=1\ne\Stab_{B_2}(Y_2^-),\Stab_{B_2}(Y_2^+)$.
\item $\Stab_{W_2}(Y_2^-)=\Stab_{W_2}(Y_2^+)=\langle(1\,3)\rangle$ and $\Stab_{W_2}(Y_2^0)=\langle\tau\rangle$.
\end{enumerate}
\end{lemma}

\begin{proof}
Recall $Y_2^-=\{0\}$, $Y_2^0=\{0,1\}$ and $Y_2^+=\{0,1,3\}$, whose complement is $\{2\}$; so $\Stab(Y_2^+)=\Stab(\{2\})$. (a) No nontrivial element of $C_4$ fixes a point, and $g_2^2=(0\,2)(1\,3)$ maps $\{0,1\}$ to $\{2,3\}$. (b) $E_2$ is regular on $\Omega_2$, and $(0\,1)(2\,3)\in E_2$ fixes $\{0,1\}$. (c) $(1\,3)\in B_2$ fixes $0$ and $2$, while $(0\,2)$, $(1\,3)$ and $(0\,2)(1\,3)$ map $\{0,1\}$ to $\{1,2\}$, $\{0,3\}$ and $\{2,3\}$. (d) The point stabilisers in $W_2$ have order $2$, and $(1\,3)$ fixes $0$ and $2$. Since $\{0,1\}=(\{0\},\{0\})$, Lemma~\ref{lem:subsets} gives $|\Stab_{W_2}(\{0,1\})|=2$, and $\tau$ fixes $\{0,1\}$.
\end{proof}

\begin{lemma}\label{lem:R3}
The group $R_3=E_2\wr C_2=\{u_1^Lu_2^R\tau^\varepsilon:u_1,u_2\in E_2,\,\varepsilon\in\{0,1\}\}\le W_3$ is transitive of order $32=2^{3\cdot2-1}$, and $\Stab_{R_3}(Y_3^0)=1\ne\Stab_{R_3}(Y_3^-),\Stab_{R_3}(Y_3^+)$.
\end{lemma}

\begin{proof}
Transitivity and the order are clear. We have $Y_3^0=(Y_2^-,Y_2^+)$, $Y_3^-=(Y_2^-,Y_2^0)$ and $Y_3^+=(Y_2^0,Y_2^+)$. Hence the assertion follows from Lemma~\ref{lem:halves} and Lemma~\ref{lem:level2}(b).
\end{proof}

\begin{definition}\label{def:admissible}
Let $j\ge2$. A subgroup $U\le W_j$ is \emph{admissible} if
\begin{enumerate}[(i)]
\item $|U|=2^{3\cdot2^{j-2}}$;
\item for each $\delta\in\{-,0,+\}$, the stabiliser $\Stab_U(Y_j^\delta)$ has order $2$; let $c_\delta$ be its generator;
\item there are homomorphisms $f_1,f_2\colon U\to\F_2$ with $f_1(c_\delta)=1$ for all $\delta$ and such that $f_2(c_-)$, $f_2(c_0)$, $f_2(c_+)$ are not all equal.
\end{enumerate}
\end{definition}

\begin{lemma}\label{lem:D8adm}
$W_2$ is admissible and transitive. Let $p_0(w)=1$ if $w$ interchanges the halves of $\Omega_2$ and $p_0(w)=0$ otherwise, and let $p_1(w)=1$ if $w$ is an odd permutation and $p_1(w)=0$ otherwise. Then one can take $f_1=p_0+p_1$ and $f_2=p_0$, and the values of $f_2$ on $c_-,c_0,c_+$ are $0,1,0$.
\end{lemma}

\begin{proof}
$|W_2|=8$, and by Lemma~\ref{lem:level2}(d) we have $c_-=c_+=(1\,3)$ and $c_0=\tau$. Both $p_0$ and $p_1$ are homomorphisms, with $p_0((1\,3))=0$, $p_1((1\,3))=1$, $p_0(\tau)=1$ and $p_1(\tau)=0$.
\end{proof}

\begin{proposition}\label{prop:fibre}
Let $j\ge2$ and let $U\le W_j$ be admissible, with $c_\delta$, $f_1$, $f_2$ as in Definition~\ref{def:admissible}, and put $w_\delta=f_2(c_\delta)$.
\begin{enumerate}[(a)]
\item Let $D'=\{u_1^Lu_2^R:u_1,u_2\in U,\ f_1(u_1)=f_1(u_2)\}$ and $U'=D'\langle\tau\rangle\le W_{j+1}$. Then $U'$ is admissible. More precisely, the stabilisers of $Y^-_{j+1}$, $Y^0_{j+1}$, $Y^+_{j+1}$ in $U'$ are generated by $c_-^Lc_0^R$, $c_-^Lc_+^R$, $c_0^Lc_+^R$, respectively, and one can take
\[f_1'(u_1^Lu_2^R\tau^\varepsilon)=f_1(u_1),\qquad f_2'(u_1^Lu_2^R\tau^\varepsilon)=f_2(u_1)+f_2(u_2).\]
The values of $f_2'$ on these three generators are $w_-+w_0$, $w_-+w_+$ and $w_0+w_+$.
\item Let $D''=\{u_1^Lu_2^R:u_1,u_2\in U,\ f_1(u_1)=f_1(u_2),\ f_2(u_1)=f_2(u_2)\}$ and $V(U)=D''\langle\tau\rangle\le W_{j+1}$. Then $|V(U)|=2^{3\cdot2^{j-1}-1}$, and $V(U)$ stabilises $(Y_j^\gamma,Y_j^\delta)$ only trivially if and only if $w_\gamma\ne w_\delta$. Hence $\Stab_{V(U)}(Y^\epsilon_{j+1})=1$ for exactly two $\epsilon\in\{-,0,+\}$.
\item Let $S\le W_i$, $i\ge2$, be of order $2^{3\cdot2^{i-2}-1}$ with $\Stab_S(Y_i^\delta)=1$ exactly for $\delta\in\{\alpha,\beta\}$, where $\alpha\ne\beta$. Then $S\wr C_2=\{s_1^Ls_2^R\tau^\varepsilon:s_1,s_2\in S\}\le W_{i+1}$ has order $2^{3\cdot2^{i-1}-1}$, and $\Stab_{S\wr C_2}(Y_{i+1}^\epsilon)=1$ for exactly one $\epsilon$, namely for the $\epsilon$ for which $Y^\epsilon_{i+1}$ equals $(Y_i^\alpha,Y_i^\beta)$ or $(Y_i^\beta,Y_i^\alpha)$.
\end{enumerate}
If $U$ (resp.\ $S$) is transitive, then so are $U'$ and $V(U)$ (resp.\ $S\wr C_2$).
\end{proposition}

\begin{proof}
(a) $D'$ is the kernel of $u_1^Lu_2^R\mapsto f_1(u_1)+f_1(u_2)$, which is surjective as $f_1(c_-)=1$, and it is normalised by $\tau$, since $\tau u_1^Lu_2^R\tau=u_2^Lu_1^R$. Hence $|U'|=2|D'|=|U|^2=2^{3\cdot2^{j-1}}$. By Lemma~\ref{lem:halves}, the stabiliser of $(Y_j^\gamma,Y_j^\delta)$ in $U'$ consists of the elements $a^Lb^R$ with $a\in\langle c_\gamma\rangle$, $b\in\langle c_\delta\rangle$ and $f_1(a)=f_1(b)$; as $f_1(c_\gamma)=f_1(c_\delta)=1$, it is $\{1,c_\gamma^Lc_\delta^R\}$. Next, $f_1'$ and $f_2'$ are homomorphisms: if $x=u_1^Lu_2^R\tau^\varepsilon$ and $y=w_1^Lw_2^R\tau^\eta$ lie in $U'$, then $xy=(u_1w_1')^L(u_2w_2')^R\tau^{\varepsilon+\eta}$, where $(w_1',w_2')$ is $(w_1,w_2)$ or $(w_2,w_1)$; since $f_1(w_1)=f_1(w_2)$, we get $f_1'(xy)=f_1'(x)+f_1'(y)$, and $f_2'$ is symmetric in the two components. Finally $f_1'(c_\gamma^Lc_\delta^R)=1$ and $f_2'(c_\gamma^Lc_\delta^R)=w_\gamma+w_\delta$. As exactly two of $w_-,w_0,w_+$ are equal, exactly one of the three sums vanishes, so $f_2'$ is not constant on the three generators.

(b) Let $\varphi=(f_1,f_2)\colon U\to\F_2^2$. Since $\varphi(c_\delta)=(1,w_\delta)$ takes both values $(1,0)$ and $(1,1)$, $\varphi$ is surjective, and $D''$, the kernel of $u_1^Lu_2^R\mapsto\varphi(u_1)+\varphi(u_2)$, has order $|U|^2/4$. It is normalised by $\tau$, so $|V(U)|=|U|^2/2=2^{3\cdot2^{j-1}-1}$. By Lemma~\ref{lem:halves}, the stabiliser of $(Y^\gamma_j,Y^\delta_j)$ consists of the $a^Lb^R$ with $a\in\langle c_\gamma\rangle$, $b\in\langle c_\delta\rangle$ and $\varphi(a)=\varphi(b)$. As $\varphi(c_\gamma),\varphi(c_\delta)\ne0$, it is trivial if and only if $\varphi(c_\gamma)\ne\varphi(c_\delta)$, that is, if and only if $w_\gamma\ne w_\delta$. Exactly two of the three pairs $\{-,0\}$, $\{-,+\}$, $\{0,+\}$ satisfy this.

(c) The order is $2|S|^2$. By Lemma~\ref{lem:halves}, the stabiliser of $(Y^\gamma_i,Y^\delta_i)$ is $\Stab_S(Y^\gamma_i)^L\Stab_S(Y^\delta_i)^R$, which is trivial if and only if $\{\gamma,\delta\}=\{\alpha,\beta\}$.

The groups $U'$ and $V(U)$ contain $u^Lu^R$ for every $u\in U$, as well as $\tau$; this proves the assertion on transitivity.
\end{proof}

\begin{corollary}\label{cor:Rj}
For every $j\ge3$ there is a transitive subgroup $R_j\le W_j$ of order $2^{3\cdot2^{j-2}-1}$ such that $\Stab_{R_j}(Y_j^0)=1$, while $\Stab_{R_j}(Y_j^-)\ne1\ne\Stab_{R_j}(Y_j^+)$.
\end{corollary}

\begin{proof}
For $j=3$ take $R_3$ from Lemma~\ref{lem:R3}. Let $U_2=W_2$, with $f_1,f_2$ as in Lemma~\ref{lem:D8adm}, and for $i\ge2$ let $U_{i+1}=U_i'$, with the homomorphisms $f_1',f_2'$ of Proposition~\ref{prop:fibre}(a). By induction, every $U_i$ is admissible and transitive, and the vector $(w_-,w_0,w_+)$ of values of $f_2$ on $c_-,c_0,c_+$ is $(0,1,0)$ for $i=2$ and $(1,0,1)$ for $i\ge3$, by the formula in Proposition~\ref{prop:fibre}(a). In both cases $w_-=w_+\ne w_0$, so by Proposition~\ref{prop:fibre}(b) the group $V(U_i)\le W_{i+1}$ is regular exactly on the $W_{i+1}$-orbits of $Y^-_{i+1}=(Y^-_i,Y^0_i)$ and $Y^+_{i+1}=(Y^0_i,Y^+_i)$. By Proposition~\ref{prop:fibre}(c), $R_{i+2}=V(U_i)\wr C_2$ has the required properties for $j=i+2\ge4$, as $Y^0_{i+2}=(Y^-_{i+1},Y^+_{i+1})$.
\end{proof}

\subsection{Proof of Theorem~\ref{thmC}}
If $q=3$ and $m=1$, take $T=P_1=\GL_1(3)$. Otherwise let $\K=\F_q$ if $q\ge5$ and $\K=\F_9$ if $q=3$, so that $P_m$ and $P(n,k)$ have the same orbits on $V_m$ and the same largest orbit size $2^{x_m}$ (Lemma~\ref{lem:sylow}); it suffices to find an irreducible $T\le P(n,k)$ of order $2^{x_m}$ with exactly one regular orbit. If $k<n$, use Proposition~\ref{prop:low}; if $n\le k\le n+2$, use Proposition~\ref{prop:mid}; and if $k\ge n+3$, use $T=T(R_{k-n})$, which has these properties and the unique regular orbit $\cO^0_k$ by Proposition~\ref{prop:reduction} and Corollary~\ref{cor:Rj}. Here irreducibility over $\F_p$ is provided by Lemma~\ref{lem:irr}. \qed

\begin{remark}
By Proposition~\ref{prop:reduction}, the group $T(V(U_{k-n-1}))$ of order $2^{x_k}$ has exactly the two regular orbits $\cO^-_k$ and $\cO^+_k$, for every $k\ge n+3$. Similarly, by Lemma~\ref{lem:level2}, at level $k=n+2$ the group $T(E_2)$ is irreducible with exactly the two regular orbits $\cO^-$ and $\cO^+$, and $T(C_4)$ is irreducible and regular on all three largest orbits.
\end{remark}

\section{Which largest orbits occur: proof of Theorem~\ref{thmE}}\label{sec:E}

In this section $n\ge2$ and $k\ge3$. For $x\in\Omega_k$ we call a map $\Omega_k\to\F_2$ of the form $x\mapsto a+bx_0+cx_1$ ($a,b,c\in\F_2$) an \emph{affine function of the two top digits}. Let $\cA_k$ be the set of all permutations of $\Omega_k$ of the form
\[x\longmapsto x\xor\Bigl(c_0+2c_1+\sum_{i=2}^{k-1}2^i\ell_i(x)\Bigr),\]
where $c_0,c_1\in\F_2$ and the $\ell_i$ are affine functions of the two top digits.

\begin{lemma}\label{lem:affine}
The set $\cA_k$ is a transitive subgroup of $W_k$ of order $2^{3k-4}$. The stabiliser in $\cA_k$ of $\{0,1,3\}$ is trivial, whereas the stabilisers of $\{0\}$ and of $\{0,1\}$ are not.
\end{lemma}

\begin{proof}
If $a,a'\in\cA_k$, then the two top digits of $x^a$ are $x_0+c_0$ and $x_1+c_1$, so an affine function of the two top digits of $x^a$ is an affine function of those of $x$; hence $aa'\in\cA_k$. As $\cA_k$ is a finite set of permutations closed under composition, it is a group. The data $c_0,c_1,\ell_2,\dots,\ell_{k-1}$ are determined by the permutation, so $|\cA_k|=4\cdot8^{k-2}=2^{3k-4}$. The digits $0,\dots,i-1$ of $x^a$ depend only on the digits $0,\dots,i-1$ of $x$, so $\cA_k\le W_k$, and $\cA_k\supseteq E_k$ is transitive.

Let $a\in\cA_k$ stabilise $X=\{0,1,3\}$. The top digits of $0$, $1$ and $3$ are $(0,0)$, $(1,0)$ and $(1,1)$, and their other digits vanish. Hence each $\ell_i$ vanishes at three of the four points of $\F_2^2$, so $\ell_i=0$. Then $a$ is the translation of the top digits by $(c_0,c_1)$, which preserves $\{(0,0),(1,0),(1,1)\}$ only if $c_0=c_1=0$. Finally, the element with $c_0=c_1=0$, $\ell_2(x)=x_0$ (respectively $\ell_2(x)=x_1$) and $\ell_i=0$ for $i\ge3$ is nontrivial and fixes $0$ (respectively $0$ and $1$).
\end{proof}

For $a,b\in\{0,1\}$ let $Q_{ab}=\{x\in\Omega_k:x_0=a,\ x_1=b\}$, the four residue classes modulo $4$, and for $d\in H$ let $\pi_{ab}(d)=\prod_{x\in Q_{ab}}d_x$. Let
\[M^{(3)}=\{d\in H:\ \pi_{00}(d)=\pi_{01}(d)=\pi_{10}(d)^{-1}=\pi_{11}(d)^{-1}\}.\]

\begin{lemma}\label{lem:M3}
$M^{(3)}$ is normalised by $W_k$, $|H:M^{(3)}|=|A|^3$, $M^{(3)}\cap A^{\{0,1,3\}}=1$, and $M^{(3)}$ contains $d^{\omega,\omega''}$ whenever $\omega$ and $\omega''$ lie in the same residue class modulo $4$.
\end{lemma}

\begin{proof}
An element of $W_k$ permutes the $Q_{ab}$, mapping the pair $\{Q_{00},Q_{01}\}$ (one half) onto itself or onto $\{Q_{10},Q_{11}\}$; the defining condition, namely that the products over the two quarters of one half are equal to some $\lambda\in A$ and those over the other half equal to $\lambda^{-1}$, is invariant under this. The map $d\mapsto(\pi_{00}(d),\pi_{01}(d),\pi_{10}(d),\pi_{11}(d))$ is onto $A^4$, and $M^{(3)}$ is the preimage of a subgroup isomorphic to $A$, so $|H:M^{(3)}|=|A|^3$. If $d\in M^{(3)}$ is supported on $\{0,1,3\}$, then $\pi_{01}(d)=1$, so $\lambda=1$ and $d_0=\pi_{00}(d)=1$, $d_1=\pi_{10}(d)=1$ and $d_3=\pi_{11}(d)=1$. The last assertion is clear.
\end{proof}

\begin{proposition}\label{prop:T3}
Let $k\in\{n+2,n+3\}$ and $T=M^{(3)}\rtimes\cA_k$. Then $|T|=2^{x_k}$, $T$ is irreducible, and $T$ has exactly one regular orbit, namely $\cO^+_{n+2}$ if $k=n+2$ and $\cO^-_{n+3}$ if $k=n+3$.
\end{proposition}

\begin{proof}
By Lemmas~\ref{lem:affine} and~\ref{lem:M3}, $|T|=|H|\,|A|^{-3}\,2^{3k-4}=2^{n2^k-3n+3k-4}$, which equals $2^{x_k}$ for $k\in\{n+2,n+3\}$ by Theorem~\ref{thm:largest}. Let $X=\{0,1,3\}$; this is $Y_2^+$ if $k=n+2$ and $Y_3^-$ if $k=n+3$. An element $da$ of $T$ ($d\in M^{(3)}$, $a\in\cA_k$) lies in $C(v_X)=A^X\rtimes\Stab_{W_k}(X)$ if and only if $a\in\Stab_{\cA_k}(X)=1$ and $d\in M^{(3)}\cap A^X=1$. So the orbit of $v_X$ is a regular orbit of $T$.

If $k=n+2$, the other largest orbits are those of $v_{\{0\}}$ and $v_{\{0,1\}}$, and they are not regular, because the permutation matrices in $\cA_k$ fixing $\{0\}$ or $\{0,1\}$ (Lemma~\ref{lem:affine}) centralise these vectors. If $k=n+3$, the other two largest orbits have zero sets with $4$ and $5$ elements, and $|M^{(3)}|\,|A|^4>|H|$, so they are not regular by Lemma~\ref{lem:count}.

Finally, $\cA_k$ is transitive, and for $\omega\ne\omega'$ we may choose $\omega''\notin\{\omega,\omega'\}$ in the residue class of $\omega$ modulo $4$, which has $2^{k-2}\ge4$ elements. Then $d^{\omega,\omega''}\in M^{(3)}$, and $T$ is irreducible by Lemma~\ref{lem:irr}.
\end{proof}

\begin{proof}[Proof of Theorem~\ref{thmE}]
As in the proof of Theorem~\ref{thmC} we may work in $P(n,k)$ with $n\ge2$. (a) Let $k=n+2$. The orbit $\cO^-=\cO_1$ is the unique regular orbit of the irreducible group $\Gamma_{n+2}$ (Proposition~\ref{prop:mid}), and $\cO^+$ that of the irreducible group of Proposition~\ref{prop:T3}. By Proposition~\ref{prop:reduction} and Lemma~\ref{lem:level2}(c), $\cO^0$ is the unique regular orbit of $T(B_2)$, which has order $2^{x_k}$ and is reducible. (b) Let $k=n+3$. The orbit $\cO^-$ is the unique regular orbit of the irreducible group of Proposition~\ref{prop:T3}, and $\cO^0$ that of the irreducible group $T(R_3)$ (Proposition~\ref{prop:reduction} and Lemma~\ref{lem:R3}).
\end{proof}

\begin{remark}
A computer calculation with a related construction, which we do not describe here, indicates that $\cO^+_{n+3}$ is also the unique regular orbit of an irreducible subgroup of order $2^{x_{n+3}}$, at least for $n\le4$. We have not pursued this further.
\end{remark}

\section{Small cases: proof of Theorem~\ref{thmD}}\label{sec:small}

\begin{proposition}\label{prop:m2}
Let $q=2^n+1\ge5$ be a Fermat prime and $P_2=H\rtimes\langle\sigma\rangle$, where $H$ is the group of diagonal matrices in $\GL_2(q)$ and $\sigma$ is the transposition matrix. Let $N=\{\diag(a,b):ab\in(\F_q^\times)^2\}$ and $M_1=N\langle\diag(\zeta,1)\sigma\rangle$. Then the subgroups of $P_2$ of order $(q-1)^2$ with exactly one regular orbit on $V_2$ are $H$ and $M_1$, and only $M_1$ is irreducible.
\end{proposition}

\begin{proof}
Here $k=1<n$, so by Theorem~\ref{thm:largest} the only largest orbit consists of the vectors $(a,b)$ with $ab\ne0$, and $C_{P_2}((1,1))=\langle\sigma\rangle$. The subgroups of order $(q-1)^2$ have index $2$, so they are normal, and by Lemma~\ref{lem:criterion} such a subgroup has exactly one regular orbit if and only if it does not contain $\sigma$. The subgroup $N$ is normal of index $4$, and $P_2/N$ is elementary abelian, generated by the images of $\diag(\zeta,1)$ and $\sigma$. As $P_2=\langle\diag(\zeta,1),\sigma\rangle$, Burnside's basis theorem gives $N=\Phi(P_2)$, so the subgroups of index $2$ are $H$, $N\langle\sigma\rangle$ and $M_1$. The elements of $M_1\setminus N$ are the $\diag(a,b)\sigma$ with $ab\notin(\F_q^\times)^2$, so $\sigma\notin M_1$. Thus $H$ and $M_1$ are the subgroups in question. $H$ is reducible, and $M_1$, which is the group of Proposition~\ref{prop:low} for $k=1$, is irreducible.
\end{proof}

\begin{proposition}\label{prop:m2q3}
Let $q=3$. The subgroups of $P_2=\GaL_1(9)\cong\SD_{16}$ of order $8$ with exactly one regular orbit on $V_2=\F_9$ are $\langle\zeta\rangle\cong C_8$ and $\langle\zeta^2,\zeta\varphi\rangle\cong Q_8$, where $\varphi\colon x\mapsto x^3$; both are irreducible.
\end{proposition}

\begin{proof}
The largest orbit is $\F_9^\times$, of size $8$. The subgroups of order $8$ are the maximal subgroups $\langle\zeta\rangle$, $\langle\zeta^2,\zeta\varphi\rangle$ and $\langle\zeta^2,\varphi\rangle$. An element $x\mapsto\zeta^ix^{3^\varepsilon}$ fixes $1$ if and only if $\zeta^i=1$; so exactly the third subgroup, which contains $\varphi$, has no regular orbit. Since $(\zeta\varphi)^2=\zeta^4=-1$, all elements outside $\langle\zeta^2\rangle$ in the second subgroup have order $4$, so it is $Q_8$. Both groups contain $\zeta^2$, which acts on $\F_9$ without eigenvalues in $\F_3$, so they are irreducible.
\end{proof}

\begin{table}[ht]
\centering
\small
\renewcommand{\arraystretch}{1.15}
\begin{tabular}{|c|c|c|c|c|p{6.3cm}|}
\hline
$q$ & $m$ & orders & classes & irr. & irreducible examples\\
\hline
$3$ & $2$ & all & $3$ & $2$ & $C_8$, $Q_8$\\
$5$ & $2$ & all & $3$ & $2$ & $D_8$; $M_1\cong\langle16,6\rangle$\\
$17$ & $2$ & all & $5$ & $4$ & $D_{32}$; $\langle128,67\rangle$, $\langle128,902\rangle$; $M_1\cong\langle256,447\rangle$\\
$3$ & $3$ & all & $34$ & $10$ & $D_{16}$; $\langle32,i\rangle$, $i\in\{6,7,11,38,44,49\}$; $\langle64,i\rangle$, $i\in\{37,45,46\}$\\
$5$ & $3$ & $2^8$ & $30$ & $16$ & $14$ with regular orbit $\cO_1$, and $2$ with regular orbit $\cO_\emptyset$\\
\hline
\end{tabular}
\medskip
\caption{Conjugacy classes in $\GL(V_m)$ of $2$-subgroups with exactly one regular orbit on $V_m$ (for $(q,m)=(5,3)$ only those of order $2^{x_3}=2^8$). Here $\langle a,b\rangle$ denotes the group with identification number $(a,b)$ in the SmallGroups library.}
\label{tab:small}
\end{table}

\begin{theorem}\label{thm:computer}
The statements of Theorem~\ref{thmD}(b) and Table~\ref{tab:small} hold.
\end{theorem}

\begin{proof}[Proof (computer-assisted)]
For each case we realised $\GL(V_m)$ as a permutation group on the nonzero vectors of $V_m$ and $P_m$ as the image of the monomial group of Section~\ref{sec:prelim} (for $q=3$, a Sylow $2$-subgroup of $\GL(V_m)$). We computed representatives of the conjugacy classes of subgroups of $P_m$ with the GAP function \texttt{ConjugacyClassesSubgroups}. For $(q,m)=(5,3)$ we computed only the classes of subgroups of order $2^8$, with the function \texttt{SubgroupsSolvableGroup} and the option \texttt{ExactSizeConsiderFunction}. For each subgroup we counted the orbits of size $|T|$ on the nonzero vectors, and we tested irreducibility with the MeatAxe. Finally, we fused the classes under conjugation in $\GL(V_m)$ with \texttt{IsConjugate}, after sorting them by their \texttt{IdGroup} and the multiset of their orbit lengths. Since every $2$-subgroup of $\GL(V_m)$ is conjugate into $P_m$, this gives all conjugacy classes. The computations were done with GAP~4.12.1 \cite{GAP}.
\end{proof}

For $(q,m)=(5,3)$ the computation shows more precisely that $18$ of the $36$ classes of subgroups of $P_3$ have $\cO_\emptyset$ and $18$ have $\cO_1$ as their unique regular orbit. All groups with regular orbit $\cO_1$ are irreducible, whereas only $2$ of the $18$ classes with regular orbit $\cO_\emptyset$ are.

\begin{remark}\label{rem:nonunique}
A subgroup of order $2^{x_m}$ with exactly one regular orbit is in general not determined by this orbit. For instance, let $k=n$ and $q\ge5$. Then both $\Gamma_n=M_n\rtimes\langle g_n\rangle$ (Proposition~\ref{prop:mid}) and $M_n\rtimes E_n$ are irreducible of order $2^{x_m}$ with unique regular orbit $\cO_1$; for the second group this follows as in the proof of Proposition~\ref{prop:mid}, since $E_n$ is regular on $\Omega_n$ and $E_n\le C(v_\emptyset)$. As their images in $W_n$ are cyclic and elementary abelian of order $2^n\ge4$, respectively, they are not conjugate in $P_m$.
\end{remark}

\begin{proof}[Proof of Theorem~\ref{thmD}]
Part~(a) is Propositions~\ref{prop:m2} and~\ref{prop:m2q3}, and part~(b) is Theorem~\ref{thm:computer}.
\end{proof}

\section{Examples below the maximal order}\label{sec:below}

\begin{example}\label{ex:wreath}
Let $q\ge5$ be a Fermat prime, let $S=(\F_q^\times)^2$, of index $2$ in $\F_q^\times$, and let $G=S\wr C_2\le\GL_2(q)$, generated by the diagonal matrices with entries in $S$ and by $\sigma$. Then $G$ is irreducible, and its unique regular orbit is $\{(a,b):ab\notin S\cup\{0\}\}$. Indeed, a vector with a zero coordinate is fixed by a nontrivial diagonal element, and $(a,b)$ with $ab\ne0$ is fixed by an element $\diag(s_1,s_2)\sigma$ if and only if $b=as_1$ and $a=bs_2$, which is possible if and only if $b/a\in S$. The $2((q-1)/2)^2=|G|$ vectors $(a,b)$ with $b/a\notin S$ form one orbit. The primitive constituent of the induced module $V$ is $S$ acting on $\F_q$, which has two regular orbits.
\end{example}

\begin{proposition}\label{prop:dihedral}
Let $D=D_{2^a}$ with $a\ge3$, let $q$ be an odd prime and $V$ a faithful irreducible $\F_qD$-module of dimension $d$, and put $Q=q^{d/2}$. Then $D$ has exactly $(Q-1)(Q+1-2^{a-1})/2^a$ regular orbits on $V$. This number is $1$ if and only if $Q=2^{a-1}+1$.
\end{proposition}

\begin{proof}
Let $C=\langle c\rangle$ be the cyclic subgroup of index $2$ and $z$ its involution. By Schur's lemma, $z$ acts as $-1$, so every nontrivial element of $C$ acts fixed-point-freely on $V\setminus\{0\}$. The $2^{a-1}$ elements of $D\setminus C$ are involutions; each such $t$ is conjugate to $tz$, which acts as $-t$, so $\dim C_V(t)=\dim C_V(tz)=d/2$. A nonzero vector fixed by two of these involutions would be fixed by their product, a nontrivial element of $C$. Hence the nonzero vectors with nontrivial stabiliser form a disjoint union of $2^{a-1}$ sets of size $Q-1$, and there are $Q^2-1-2^{a-1}(Q-1)=(Q-1)(Q+1-2^{a-1})$ vectors in regular orbits. Finally, $(Q-1)(Q+1-2^{a-1})=2^a$ forces $Q-1=2^s$ and $2^s+2-2^{a-1}=2^{a-s}$, so that $s=a-1$.
\end{proof}

For instance, $D_{2(q-1)}=\langle\diag(\zeta,\zeta^{-1}),\sigma\rangle\le\GL_2(q)$ has exactly one regular orbit on $\F_q^2$ for every Fermat prime $q\ge5$; and $D_{16}$, acting on $\F_9^2$ through $\langle\diag(\lambda,\lambda^{-1}),\sigma\rangle$ with $\lambda$ of order $8$ and regarded over $\F_3$, has exactly one regular orbit on $\F_3^4$. Similarly, a generalised quaternion group $Q_{2^a}$ acts semiregularly on $V\setminus\{0\}$ for every faithful irreducible module $V$, because its unique involution acts as $-1$. So it has exactly one regular orbit if and only if $|V|=2^a+1$, that is, for $Q_8$ acting on $\F_3^2$.

\section{Open problems}\label{sec:open}

\begin{question}
Let $k\ge n+4$. Is $\cO^-_k$ or $\cO^+_k$ the unique regular orbit of a subgroup of $P_m$ of order $2^{x_m}$? Our constructions always produce $\cO^0_k$.
\end{question}

\begin{question}
Let $k=n+2$. Is $\cO^0$ the unique regular orbit of an irreducible subgroup of $P_m$ of order $2^{x_m}$? Among the groups $T(R)$ of Proposition~\ref{prop:reduction}, only the reducible group $T(B_2)$ has this property.
\end{question}

\begin{question}
Is every subgroup of $P_m$ of order $2^{x_m}$ with exactly one regular orbit, for $k\ge n+3$, conjugate to a group of the form $T(R)$? If so, the classification of these subgroups reduces to that of subgroups of $W_{k-n}$ acting on subsets of $\Omega_{k-n}$.
\end{question}

\begin{question}
Classify the subgroups of order $2^{x_m}$ with exactly one regular orbit for further small cases, for instance $(q,m)\in\{(3,4),(5,4),(17,3)\}$, including the semilinear subgroups of $P_m$ for $q=3$. Explain the even split between the two largest orbits observed for $(q,m)=(5,3)$, which parallels the case $p=3$ in \cite{KR}.
\end{question}

\begin{question}
Classify all $2$-groups with exactly one regular orbit on a faithful irreducible module, including those whose regular orbit is not of the largest possible size. Example~\ref{ex:wreath} and Proposition~\ref{prop:dihedral} suggest organising this by the primitive constituent.
\end{question}

\section*{Acknowledgement}
The first author was supported by a grant from the Simons Foundation (SFI-MPS-TSM-00014153 to Thomas M. Keller).

\section*{AI statement}
This paper is a natural sequel to \cite{KR} and had been on the authors' ``to do''-list to follow the completion of .\cite{KR}. 
With the advent of AI, its results were now obtained with substantial help from Claude by Anthropic. Before using  this tool,
the authors knew most of Theorem A and had some ideas towards Theorems B and C. . Claude helped with Section 4, proposed the constructions in Sections~\ref{sec:C} and~\ref{sec:E}, carried out the computations reported in Section~\ref{sec:small} together with further checks of the constructions, and drafted the text. The authors have reviewed and edited the content and take full responsibility for the mathematical and narrative integrity of the paper.

\end{document}